\documentclass[preprint,1p,times,sort&compress]{elsarticle}
\usepackage{amsmath,amsfonts,mathrsfs}
\usepackage{amssymb}

\usepackage{amsthm}
\newtheorem{theorem}{Theorem}[section]
\newtheorem{lemma}{Lemma}[section]
\newtheorem{example}{Example}[section]

\newtheorem{rem}{Remark}[section]

\numberwithin{equation}{section}
\def\vector{\mathrm{vec}}

\journal{
        }

\begin{document}
\begin{frontmatter}
\title{ Perturbation analysis of the matrix equation $X-\sum\limits_{i=1}^{m}A_{i}^*X^{p_{i}}A_{i}=Q$
 \tnoteref{label1} }
\tnotetext[label1]{The work was supported in part by Natural Science Foundation of Shandong Province (ZR2012AQ004).}
\author
       {Jing Li
        }
\ead{xlijing@sdu.edu.cn}

\address
               {School of Mathematics and Statistics, Shandong University at Weihai, Weihai 264209, P.R. China}

\begin{abstract}
Consider the nonlinear matrix equation $X-\sum\limits_{i=1}^{m}A_{i}^*X^{p_{i}}A_{i}=Q$ with $p_{i}>0$. Sufficient and necessary conditions for the existence of
positive definite solutions to the equation with
$p_{i}>0$ are derived. Two perturbation bounds for the unique solution to the equation with $0<p_{i}<1$ are evaluated. The
backward error of an approximate solution for the unique solution to the equation with $0<p_{i}<1$ is given.
Explicit expressions of the condition number for the equation with $0<p_{i}<1$ are obtained. The theoretical results are illustrated by numerical examples.
\end{abstract}

\end{frontmatter}

\section{Introduction}
In this paper the nonlinear matrix equation
\begin{equation}\label{eq1}                                   
X-\sum\limits_{i=1}^{m}A_{i}^*X^{p_{i}}A_{i}=Q
\end{equation}
is investigated, where $A_{1}, A_{2}, \ldots , A_{m}$ are $n\times n$ complex matrices, $m$ is a positive integer, $p_{i}>0 \;(i=1,2, \cdots m)$ and $Q$ is a positive definite matrix. Here, $A_{i}^{*}$ denotes
the conjugate transpose of the matrix $A_{i}$.

When $m>1$, Eq.(\ref{eq1}) is recognized as playing an important role in
solving a system of linear equations. For example, in many physical calculations,
one must solve the system of linear equation
$$Mx=f,$$
where $$M=\left(\begin{array}{ccccc}
I & 0& \cdots& 0& A_{1}\\
0 & I& \cdots& 0& A_{2}\\
\vdots & \vdots & \ddots & \vdots & \vdots\\
0 & 0& \cdots & I& A_{m}\\
A_{1}^{*} & A_{2}^{*}& \cdots & A_{m}^{*}& -Q
\end{array}\right)$$
arises in a finite difference approximation to an elliptic partial differential equation (for more information,
refer to \cite{s7} ).
We can rewrite $M$ as $M=\widetilde{M}+D$, where
$$\widetilde{M}=\left(\begin{array}{ccccc}
X^{-p_{1}} & 0& \cdots& 0& A_{1}\\
0 & X^{-p_{2}}& \cdots& 0& A_{2}\\
\vdots & \vdots & \ddots & \vdots & \vdots\\
0 & 0& \cdots & X^{-p_{m}}& A_{m}\\
A_{1}^{*} & A_{2}^{*}& \cdots & A_{m}^{*}& -Q
\end{array}\right),\;\;\;
D=\left(\begin{array}{ccccc}
I-X^{-p_{1}} & 0& \cdots& 0& 0\\
0 & I-X^{-p_{2}}& \cdots& 0& 0\\
\vdots & \vdots & \ddots & \vdots & \vdots\\
0 & 0& \cdots & I-X^{-p_{m}}& 0\\
0 & 0& \cdots & 0 & 0
\end{array}\right).$$
$\widetilde{M}$ can be factored as $$\widetilde{M}=\left(\begin{array}{ccccc}
-I & 0& \cdots& 0& 0\\
0 & -I& \cdots& 0& 0\\
\vdots & \vdots & \ddots & \vdots & \vdots\\
0 & 0& \cdots & -I & 0\\
-A_{1}^{*}X^{p_{1}} & -A_{2}^{*}X^{p_{2}}& \cdots & -A_{m}^{*}X^{p_{m}}& -I
\end{array}\right)
\left(\begin{array}{ccccc}
-X^{-p_{1}} & 0& \cdots& 0& -A_{1}\\
0 & -X^{-p_{2}}& \cdots& 0& -A_{2}\\
\vdots & \vdots & \ddots & \vdots & \vdots\\
0 & 0& \cdots & -X^{-p_{m}}& -A_{m}\\
0 & 0& \cdots & 0& X
\end{array}\right)$$ if and only if $X$ is a solution of equation
$X-\sum\limits_{i=1}^{m}A_{i}^*X^{p_{i}}A_{i}=Q$.
When $m=1,$ this type of nonlinear matrix
equations arises in ladder networks, dynamic
programming, control theory, stochastic filtering, statistics and
so forth \cite{s1,s2,s3,s4,s5,s6}.

For the similar equations $X\pm A^{*}X^{-p}A=Q$, $X^{s}\pm A^{*}X^{-t}A=Q$ and $X+\sum\limits_{i=1}^{m}A_{i}^*X^{-1}A_{i}=I$ , there were many contributions
in the literature to the theory, applications and numerical
solutions \cite{s30,s8,s9,s10,s31,s12,s13,s22,s23,s24,s25,s26,s14,s15,s16,s17,s18,s19,s20,s21,s27,s28,s39}.
Jia and Gao \cite{s38} derived two perturbation estimates for the solution of the equation $X-A^{*}X^{q}A=Q$ with $0<q<1$. In addition, Duan et al. \cite{s29} proved that the equation
$X-\sum\limits_{i=1}^{m}A_{i}^*X^{\delta_{i}}A_{i}=Q\;(0<|\delta_{i}|<1)$
has a unique positive definite solution. They also proposed an iterative method
for obtaining the unique positive definite solution. However, to our best knowledge, there has been no perturbation analysis for Eq.(\ref{eq1}) with $m>1$ in the known literatures.

The rest of the paper is organized as follows. In Section 2, some preliminary lemmas are given.
In Section 3 , sufficient and necessary conditions for Eq. (\ref{eq1})
existing positive definite solutions are derived.
In Section 4 , two perturbation bounds for
 the unique solution to Eq.(\ref{eq1}) with $0<p_{i}<1$ are given. Furthermore, in Section 5,  we obtain the backward error of an approximate solution for Eq.(\ref{eq1}) with $0<p_{i}< 1$. In Section 6, we also discuss the condition number of the unique solution to Eq.(\ref{eq1}). Finally, several numerical
  examples are presented in Section 7.

We denote by $\mathcal{C}^{n\times n}$ the set of $n\times n$
complex matrices, by $\mathcal{H}^{n\times n}$ the set of $n\times n$
Hermitian matrices, by $I$ the identity matrix, by $\textbf{i}$ the imaginary unit, by $\|\cdot\|$ the spectral
norm, by $\|\cdot\|_{F}$ the Frobenius norm and by $\lambda_{\max}(M)$ and
$\lambda_{\min}(M)$ the maximal and minimal eigenvalues of $M$,
respectively. For $A=(a_{1},\dots, a_{n})=(a_{ij})\in
\mathcal{C}^{n\times n}$ and a matrix $B$, $A\otimes B=(a_{ij}B)$
is a Kronecker product, and $\vector A$
is a vector defined by $\vector A=(a_{1}^{T},\dots,
a_{n}^{T})^{T}$. For $X, Y\in\mathcal{H}^{n\times n}$, we write $X\geq Y$ (resp. $X>Y)$ if
$X-Y$ is Hermitian positive semi-definite (resp. definite).
\section{Preliminaries}
\begin{lemma}\label{lem1}                   
\cite{s33}. If $A\geq B>0$ and $0\leq \gamma\leq 1,$ then $A^{\gamma}\geq B^{\gamma}.$
\end{lemma}

\begin{lemma}\label{lem2}                   
\cite{s38}. For any Hermitian positive definite matrix $X$ and Hermitian matrix $\Delta X$, we have
 \begin{enumerate}
   \item[(i)]$X^{q}=\displaystyle\frac{\sin q\pi}{\pi}\int_{0}^{\infty}X^{\frac{1}{2}}(\lambda I+X)^{-1}X^{\frac{1}{2}}\lambda^{q-1}d\lambda,\;\;0<q<1;$
   \item[(ii)]$X^{q}=\displaystyle\frac{\sin q\pi}{(1-q)\pi}\int_{0}^{\infty}X^{\frac{1}{2}}(\lambda I+X)^{-1}X(\lambda I+X)^{-1}X^{\frac{1}{2}}\lambda^{q-1}d\lambda,\;\;0<q<1.$
 \end{enumerate}
   In addition, if $X+\Delta X\geq (1/\nu) X>0$, then
\begin{enumerate}
\item[(iii)]$ \|X^{-\frac{1}{2}}A^{*}((X+\Delta X)^{q}-X^{q})AX^{-\frac{1}{2}}\|
\leq (1-q)(\|X^{-\frac{1}{2}}\Delta X X^{-\frac{1}{2}}\|+\nu \|X^{-\frac{1}{2}}\Delta X X^{-\frac{1}{2}}\|^{2})\|X^{\frac{q}{2}}AX^{-\frac{1}{2}}\|^{2}.$
\end{enumerate}
\end{lemma}

\begin{lemma}\label{lem3}                 
\cite{s29}. The matrix equation
$X-\sum\limits_{i=1}^{m}A_{i}^{*}X^{\delta_{i}}A_{i}=Q\;(0<|\delta_{i}|<1)$ always has a unique positive definite solution $X$. The matrix sequence $X_{k}:$
\begin{equation}\label{eq13}
X_{s+m+1}=Q+\sum\limits_{i=1}^{m}A_{i}^{*}X^{\delta_{i}}_{s+i}A_{i},\;\;s=0, 1, 2, \cdots ,
\end{equation}
converges to the unique positive definite solution $X$ for arbitrary initial positive definite matrices $X_{1}, X_{2}, \cdots, X_{m}.$
\end{lemma}
\section{Positive definite solutions of the matrix Eq.(1.1)}
In this section, sufficient and necessary conditions for the existence of positive definite solutions of Eq.(\ref{eq1}) are obtained.

\begin{theorem}     \label{thm1}                                                                     
Eq.(\ref{eq1}) has a positive definite solution $X$ if and only if the coefficient matrices $A_{i}$ can be factored as
\begin{equation}\label{eq2}                                           
A_{i}=(W^{*}W)^{-p_{i}/2}Y_{i}(W^{*}W)^{1/2},
\end{equation}
where $W,$ $Y_{i}\;(i=1,2, \cdots, m)$ are nonsingular matrices, $Z=Q^{\frac{1}{2}}(W^{*}W)^{-\frac{1}{2}}$  and
$\left(\begin{array}{c}
Z\\
 Y_{1}\\
  \vdots \\
 Y_{m} \\
  \end{array}
   \right)$ is column orthonormal. In this case $X=W^{*}W$ is a solution of Eq.(\ref{eq1}).
\end{theorem}
\begin{proof}
{If Eq.(\ref{eq1}) has a positive definite solution $X$, then there exists a  nonsingular matrix $W,$ s.t. $X=W^{*}W.$ Furthermore Eq.(\ref{eq1}) can be rewritten as
$$W^{*}W-\sum\limits_{i=1}^{m}A_{i}^{*}(W^{*}W)^{p_{i}}A_{i}=Q,$$ which implies that
\begin{equation}     \label{eq15}                                                               
(W^{*}W)^{-\frac{1}{2}}Q(W^{*}W)^{-\frac{1}{2}}+\sum\limits_{i=1}^{m}(W^{*}W)^{-\frac{1}{2}}A_{i}^{*}(W^{*}W)^{\frac{p_{i}}{2}}
(W^{*}W)^{\frac{p_{i}}{2}}A_{i}(W^{*}W)^{-\frac{1}{2}}=I.
\end{equation}
Let $Y_{i}=(W^{*}W)^{\frac{p_{i}}{2}}A_{i}(W^{*}W)^{-\frac{1}{2}},$ then $A_{i}=(W^{*}W)^{\frac{-p_{i}}{2}}Y_{i}(W^{*}W)^{\frac{1}{2}}.$ Moreover, by $Z=Q^{\frac{1}{2}}(W^{*}W)^{-\frac{1}{2}},$ (\ref{eq15}) turns into
$$\left(\begin{array}{c}
Z\\
 Y_{1}\\
  \vdots \\
 Y_{m} \\
  \end{array}
   \right)^{*}
   \left(\begin{array}{c}
Z\\
 Y_{1}\\
  \vdots \\
 Y_{m} \\
  \end{array}
   \right)=I,$$ which means that $\left(\begin{array}{c}
Z\\
 Y_{1}\\
  \vdots \\
 Y_{m} \\
  \end{array}
   \right) $ is column orthonormal.

   Conversely, suppose that $A_{i}$ has the decomposition (\ref{eq2}). Let $X=W^{*}W$. Then
   \begin{eqnarray*}
   X-\sum\limits_{i=1}^{m}A_{i}^{*}X^{p_{i}}A_{i} &=&
W^{*}W-\sum\limits_{i=1}^{m}(W^{*}W)^{\frac{1}{2}}Y_{i}^{*}(W^{*}W)^{-\frac{p_{i}}{2}}(W^{*}W)^{p_{i}}
(W^{*}W)^{-\frac{p_{i}}{2}} Y_{i}(W^{*}W)^{\frac{1}{2}}\\
     &=&  W^{*}W-\sum\limits_{i=1}^{m}(W^{*}W)^{\frac{1}{2}}Y_{i}^{*}Y_{i}(W^{*}W)^{\frac{1}{2}}\\
      &=& (W^{*}W)^{\frac{1}{2}}(I-\sum\limits_{i=1}^{m}Y_{i}^{*}Y_{i}) (W^{*}W)^{\frac{1}{2}}=
     (W^{*}W)^{\frac{1}{2}}Z^{*}Z (W^{*}W)^{\frac{1}{2}} \\
     &=&(W^{*}W)^{\frac{1}{2}}(W^{*}W)^{-\frac{1}{2}}Q^{\frac{1}{2}}Q^{\frac{1}{2}}
     (W^{*}W)^{-\frac{1}{2}}(W^{*}W)^{\frac{1}{2}}=Q,
   \end{eqnarray*}
    that is $X=W^{*}W$ being a positive definite solution of Eq.(\ref{eq1}).
}\end{proof}

\begin{theorem}     \label{thm2}                 
If $A_{1}, A_{2}, \cdots, A_{m}$ are invertible and $Q\in\mathcal{H},$ then Eq. (\ref{eq1}) has a positive definite solution $X$
if and only if $A_{i}$ can be factored as
\begin{equation}\label{eq3}                                            
A_{i}=(U^{*}MU)^{-\frac{p_{i}}{2}}V_{i}NU,
\end{equation}
where $N>0$, $U$ is unitary and $M>UQU^{*}$ is diagonal, $M-N^{2}=UQU^{*}$ and $\left(\begin{array}{c}
V_{1}\\
 V_{2}\\
  \vdots \\
 V_{m} \\
  \end{array}
   \right)
  $ is column orthonormal.
 In this case, X=$U^{*}MU$ is a solution of Eq.(\ref{eq1}).
\end{theorem}
\begin{proof}
If Eq.(\ref{eq1}) has a positive definite solution $X$, then $X$ can be factored as $X=U^{*}MU,$ where
$U$ is unitary and $M$ is diagonal. Therefore Eq.(\ref{eq1}) can be rewritten as
$$U^{*}MU-\sum\limits_{i=1}^{m}A_{i}^{*}(U^{*}MU)^{p_{i}}A_{i}=Q,$$ which implies that
$$M-UQU^{*}=\sum\limits_{i=1}^{m}UA_{i}^{*}(U^{*}MU)^{p_{i}}A_{i}U^{*}$$ and
\begin{equation} \label{eq4}                                      
\sum\limits_{i=1}^{m}(M-UQU^{*})^{-\frac{1}{2}}UA_{i}^{*}(U^{*}MU)^{p_{i}}A_{i}U^{*}(M-UQU^{*})^{-\frac{1}{2}}=I.
\end{equation}
Let $N=(M-UQU^{*})^{\frac{1}{2}}$ and $V_{i}=(U^{*}MU)^{\frac{p_{i}}{2}}A_{i}U^{*}(M-UQU^{*})^{-\frac{1}{2}}.$
Then $A_{i}=(U^{*}MU)^{-\frac{p_{i}}{2}}V_{i}NU$ and $M-N^{2}=UQU^{*}. $ Eq.(\ref{eq4}) turns into
$\sum\limits_{i=1}^{m}V_{i}^{*}V_{i}=I,$ which means that $\left(\begin{array}{c}
V_{1}\\
 V_{2}\\
  \vdots \\
 V_{m} \\
  \end{array}
   \right)
  $ is column orthonormal.

  Conversely, suppose that $A_{i}$ has the decomposition (\ref{eq3}). Let $X=U^{*}MU.$ Then
  \begin{eqnarray*}
X-\sum\limits_{i=1}^{m}A_{i}^{*}X^{p_{i}}A_{i}&=&U^{*}MU-
\sum\limits_{i=1}^{m}U^{*}N^{*}V_{i}^{*}(U^{*}MU)^{-\frac{p_{i}}{2}}(U^{*}MU)^{p_{i}}(U^{*}MU)^{-\frac{p_{i}}{2}}V_{i}NU\\
&=& U^{*}MU-
\sum\limits_{i=1}^{m}U^{*}N^{*}V_{i}^{*}V_{i}NU \\
&=&  U^{*}(M-N^{2})U=Q,
  \end{eqnarray*} that is $X=U^{*}MU$ is a solution of Eq.(\ref{eq1}).
\end{proof}

\begin{theorem}\label{thm3}                                      
If $X$ is a solution of Eq.(\ref{eq1}) with $0<p_{i}<1$, then $$X\geq\left(\lambda_{\min}(Q)+
\sum\limits_{i=1}^{m}\displaystyle\lambda_{\min(A_{i}^{*}A_{i})}\lambda_{\min}^{p_{i}}(Q)\right)I= \beta I.$$
\end{theorem}
\begin{proof}
By Lemma \ref{lem3}, Eq.(\ref{eq1}) with $0<p_{i}<1$ always has a unique positive definite solution $X$. Then $X>0$, it follows that $X^{p_{i}}>0.$ Therefore $X\geq Q.$ By Lemma \ref{lem1} and Eq.(\ref{eq1}), we have $X\geq Q+\sum\limits_{i=1}^{m}A_{i}^{*}Q^{p_{i}}A_{i}\geq \left(\lambda_{min}(Q)+\sum\limits_{i=1}^{m}\displaystyle\lambda_{\min(A_{i}^{*}A_{i})}\lambda_{\min}^{p_{i}}(Q)\right)I=\beta I.$
\end{proof}

\section{Perturbation bounds of Eq.(\ref{eq1}) with $0<p_{i}<1$}

Here the perturbed equation
\begin{equation}\label{eq7}                                                                
\widetilde{X}-\sum\limits_{i=1}^{m}\widetilde{A_{i}}^*\widetilde{X}^{p_{i}}\widetilde{A_{i}}=\widetilde{Q},
\;\;0<p_{i}<1,
\end{equation}
is considered,
where $\widetilde{A_{i}}$ and $\widetilde{Q}$ are  small perturbations
of $A_{i}$ and $Q$ in Eq.(\ref{eq1}), respectively. We assume that $X$
and $\widetilde{X}$ are the solutions of Eq.(\ref{eq1}) and
Eq.(\ref{eq7}), respectively. Let $\Delta X=\widetilde{X}-X$, $\Delta Q=\widetilde{Q}-Q$ and $\Delta A_{i}=\widetilde{A_{i}}-A_{i}$.

In this section two perturbation bounds for the
solution of Eq.(\ref{eq1})  with $0<p_{i}<1$ are developed. The relative perturbation bound in
Theorem \ref{thm6} does not depend any knowledge of the actual solution $X$
of Eq.(\ref{eq1}). Furthermore, a sharper perturbation bound in
Theorem \ref{thm9} is derived.

The next theorem generalizes Theorem 4 in \cite{s38}
with $m=1$, $\|\Delta Q\|=0$ to arbitrary integer $m\geq 1,$ $\|\Delta Q\|>0$.

\begin{theorem}\label{thm6}                           
 Let $b=\beta+\|\Delta Q\|-\sum\limits_{i=1}^{m}(1-{p_{i}})\beta^{p_{i}}\|A_{i}\|^{2},
s=\sum\limits_{i=1}^{m}\beta^{p_{i}}\|\Delta A_{i}\|\;(2\|A_{i}\|+\|\Delta A_{i}\|)$. If
\begin{equation}\label{eq8}                              
0<b<2(\beta-s) \;\;{and} \;\;b^{2}-4(\beta-s)(s+\|\Delta Q\|)\geq 0,
 \end{equation}
then
\begin{equation}\label{eq25}                            
\frac{\|\widetilde{X}-X\|}{\|X\|}\,\leq
\varrho\,\sum\limits_{i=1}^{m}\|\Delta A_{i}\|+\omega \|\Delta Q\|\equiv \xi_{1},
\end{equation}
where
$$\varrho=\displaystyle\frac{2s}{\sum\limits_{i=1}^{m}\|\Delta A_{i}\|(b+\sqrt{b^{2}-4(\beta-s)(s+\|\Delta Q\|)})},\;\;\;\;\omega=\frac{2}{b+\sqrt{b^{2}-4(\beta-s)(s+\|\Delta Q\|)}}.$$

\end{theorem}

\begin{proof}  Let
$$\Omega=\{\Delta X \in\mathcal{H}^{n\times n}:\,\,\|
X^{-1/2}\Delta X X^{-1/2} \|\leq \varrho\,\sum\limits_{i=1}^{m}\|\Delta
A_{i}\|+\omega\|\Delta Q\|\;\}.$$ Obviously, $\Omega$ is a nonempty
bounded convex closed set. Let $$f(\Delta
X)=\sum\limits_{i=1}^{m}(\widetilde{A_{i}}^*(X+\Delta
X)^{p_{i}}\widetilde{A_{i}}-A_{i}^*X^{p_{i}}A_{i})+\Delta Q,\;\;\Delta
X\in\Omega.$$ Evidently, $f: \Omega\mapsto\mathcal{H}^{n\times n}$
is continuous. We will prove that $f(\Omega)\subseteq\Omega$.

For every $\Delta X\in\Omega,$ it follows  $\|X^{-1/2}\Delta XX^{-1/2}\|\leq
\varrho\,\sum\limits_{i=1}^{m}\|\Delta A_{i}\|+\omega \|\Delta Q\|.$ Thus
$$(\varrho \,\sum\limits_{i=1}^{m}\|\Delta A_{i}\|+\omega\|\Delta Q\| )I\geq X^{-1/2}\Delta XX^{-1/2}\geq (-\varrho \,\sum\limits_{i=1}^{m}\|\Delta A_{i}\|-\omega\|\Delta Q\| )I,$$
$$(1+\varrho \sum\limits_{i=1}^{m}\|\Delta A_{i}\|+\omega\Delta Q\|\|)X\geq X+\Delta X\geq (1-\varrho \sum\limits_{i=1}^{m}\|\Delta A_{i}\|-\omega\Delta Q\|\|)X.$$

According to (\ref{eq8}) and (\ref{eq25}), we have $$\varrho\sum\limits_{i=1}^{m}\|\Delta A_{i}\|+\omega\|\Delta Q\|=\frac{2(\|\Delta Q\|+s)}{b+\sqrt{b^{2}-4(\beta-s)(s+\|\Delta Q\|)}}\leq \frac{2(\|\Delta Q\|+s)}{b}\leq\frac{b}{2(\beta-s)}<1.$$ Therefore $$(1-\varrho\sum\limits_{i=1}^{m}\|\Delta A_{i}\|-\omega\|\Delta Q\|)X>0.$$
From Lemma \ref{lem2} and Theorem \ref{thm3}, it
follows that
\begin{eqnarray}\label{eq16}                                                         
&&\left \| X^{-\frac{1}{2}}\left[\sum\limits_{i=1}^{m}A_{i}^*\left((X+\Delta
X\right)^{p_{i}}-X^{p_{i}})A_{i}\right]X^{-\frac{1}{2}}\right\|\nonumber \\
&\leq &\left(\| X^{-\frac{1}{2}}\Delta X X^{-\frac{1}{2}}\| +
\frac{\|X^{-\frac{1}{2}}\Delta X
X^{-\frac{1}{2}}\|^{2}}{1-\varrho\;\sum\limits_{i=1}^{m}\|\Delta
A_{i}\|-\omega\|\Delta Q\|}\right)(\sum\limits_{i=1}^{m}(1-p_{i})\|X^{\frac{p_{i}}{2}}A_{i}X^{-\frac{1}{2}}\|^{2})\\
&\leq &\left(\| X^{-\frac{1}{2}}\Delta X X^{-\frac{1}{2}}\| +
\frac{\|X^{-\frac{1}{2}}\Delta X
X^{-\frac{1}{2}}\|^{2}}{1-\varrho\;\sum\limits_{i=1}^{m}\|\Delta
A_{i}\|-\omega\|\Delta Q\|}\right)(\sum\limits_{i=1}^{m}\frac{1-p_{i}}{\beta^{1-p_{i}}}\|A_{i}\|^{2}).\nonumber
\end{eqnarray}

Therefore
\begin{eqnarray*}
&&\left\| X^{-\frac{1}{2}}f(\Delta
X)X^{-\frac{1}{2}}\right\|\\
&=&\left\| X^{-\frac{1}{2}}\left[\sum\limits_{i=1}^{m}\widetilde{A_{i}}^*(X+\Delta
X)^{p_{i}}\widetilde{A_{i}}-A_{i}^{*}X^{p_{i}}A_{i}\right]X^{-\frac{1}{2}}+X^{-\frac{1}{2}}\Delta QX^{-\frac{1}{2}}\right\|\\
&\leq &\left\| \sum\limits_{i=1}^{m}X^{-\frac{1}{2}}A_{i}^*((X+\Delta
X)^{p_{i}}-X^{p})A_{i}X^{-\frac{1}{2}}\right\|+\|X^{-\frac{1}{2}}\Delta QX^{-\frac{1}{2}}\|\\
&&+\left\|\sum\limits_{i=1}^{m}X^{-\frac{1}{2}}\left[\Delta A_{i}^{*}(X+\Delta X)^{p_{i}}(A_{i}+\Delta A_{i})+A_{i}^{*}(X+\Delta X)^{p_{i}}\Delta A_{i}\right]X^{-\frac{1}{2}}\right\|
\\
&\leq & \left(\| X^{-\frac{1}{2}}\Delta X
X^{-\frac{1}{2}}\| +
\frac{\|X^{-\frac{1}{2}}\Delta
X X^{-\frac{1}{2}}\|^{2}}{1-\varrho\sum\limits_{i=1}^{m}\|\Delta
A_{i}\|-\omega\|\Delta Q\|}\right)\left(\sum\limits_{i=1}^{m}\frac{1-p_{i}}{\beta^{1-p_{i}}}\|A_{i}\|^{2}\right)\\
&&+\sum\limits_{i=1}^{m}\frac{\|\Delta A_{i}\|(2\| A_{i}\|+\|\Delta
A_{i}\|)}{\beta^{1-p_{i}}}(1+\varrho\sum\limits_{i=1}^{m}\|\Delta
A_{i}\|+\omega\|\Delta Q\|)+\frac{\|\Delta Q\|}{\beta}\\
&\leq & \left(\xi_{1}+\frac{\xi_{1}^{2}}{1-\xi_{1}}\right)\left(
\sum\limits_{i=1}^{m}\frac{1-p_{i}}{\beta^{1-p_{i}}}\|A_{i}\|^{2}\right)+
\frac{s}{\beta}(1+\xi_{1})+\frac{\|\Delta Q\|}{\beta}\\
&=&\xi_{1}.
\end{eqnarray*}
That is $f(\Omega)\subseteq\Omega.$ By Brouwer's fixed point theorem, there exists a $\Delta
X\in\Omega$ such that $f(\Delta X)=\Delta X$. Moreover, by Lemma
\ref{lem3}, we know that $X$ and $\widetilde{X}$ are the unique
solutions to Eq.(\ref{eq1}) and Eq.(\ref{eq7}), respectively. Then
\begin{eqnarray*}
\frac{\|\widetilde{X}-X\|}{\|X\|} & = & \frac{\|\Delta
X\|}{\|X\|}=\frac{\| X^{1/2}(X^{-1/2}\Delta X
X^{-1/2})X^{1/2}\|}{\|X\|}\\
& \leq &\|X^{-1/2}\Delta X X^{-1/2}\|\leq \varrho\sum\limits_{i=1}^{m}\|\Delta
A_{i}\|+\omega\|\Delta Q\|.
\end{eqnarray*}
\end{proof}
\begin{rem}
With $$\varrho\sum\limits_{i=1}^{m}\|\Delta
A_{i}\|+\omega\|\Delta Q\|=\frac{2(\sum\limits_{i=1}^{m}\|\Delta A_{i}\|(2\|A_{i}\|+\|\Delta A_{i}\|)+\|\Delta Q\|)}{b+\sqrt{b^{2}-4(\beta-s)(s+\|\Delta Q\|)}},$$ we get
$\varrho\sum\limits_{i=1}^{m}\|\Delta
A_{i}\|+\omega\|\Delta Q\|\rightarrow 0$ for $\Delta Q\rightarrow 0,$ $\|\Delta A_{i}\|\rightarrow 0\;(i=1,2, \cdots, m).$ Therefore Eq.(\ref{eq1}) is well-posed.
\end{rem}

Next, a sharper perturbation estimate is derived.

Subtracting (\ref{eq1}) from (\ref{eq7}) we have
\begin{equation}\label{eq18}                                 
  \Delta X+\sum\limits_{i=1}^{m}\frac{\sin p_{i}\pi}{\pi}\int^{\infty}_{0}[(\lambda I+X
)^{-1}X^{\frac{1}{2}}A_{i}]^{*}\Delta X[(\lambda I+X
)^{-1}X^{\frac{1}{2}}A_{i}]\lambda^{p_{i}-1}d\lambda=E+h(\Delta X),
\end{equation}
where
\begin{eqnarray}\label{eq5}                                                                     
&& B_{i}=X^{p_{i}}A_{i}, \;\;i=1, 2, \cdots, m, \nonumber\\
&& E=\sum\limits_{i=1}^{m}(B_{i}^{*}\Delta A_{i}+\Delta A_{i}^{*}B_{i})
+\sum\limits_{i=1}^{m}\Delta A_{i}^{*}X^{p_{i}}\Delta A_{i}+\Delta Q, \nonumber\\
&&h(\Delta X)=\sum\limits_{i=1}^{m}\left[A_{i}^{*}Z_{i}(\Delta X)A_{i}-\widetilde{A}_{i}^{*}V_{i}(\Delta X)\Delta A_{i}
-\Delta A_{i}^{*}V_{i}(\Delta X)A_{i}\right],\\
&&Z_{i}(\Delta X)=\frac{\sin p_{i}\pi}{\pi}\int^{\infty}_{0}X^{\frac{1}{2}}(\lambda I+X)^{-1}\Delta X(\lambda I+X+\Delta X)^{-1}\Delta X (\lambda I+X)^{-1}X^{\frac{1}{2}}\lambda^{p_{i}-1}d\lambda,\nonumber\\
&& V_{i}(\Delta X)=\frac{\sin p_{i}\pi}{\pi}\int^{\infty}_{0}X^{\frac{1}{2}}(\lambda I+X)^{-1}\Delta X(\lambda I+X+\Delta X)^{-1}X^{\frac{1}{2}}\lambda^{p_{i}-1}d\lambda.\nonumber
\end{eqnarray}

\begin{lemma}  \label{lem5}                                                
Let $\sum\limits_{i=1}^{m}\displaystyle\frac{\|A_{i}\|^{2}}{\beta^{1-p_{i}}}<1.$ Then the linear operator $\mathbf{L}:\mathcal{H}^{n\times
n}\rightarrow\mathcal{H}^{n\times n}$ defined by
\begin{equation}\label{eq24}
 \mathbf{L}W=W+\sum\limits_{i=1}^{m}\frac{\sin p_{i}\pi}{\pi}\int^{\infty}_{0}[(\lambda I+X
)^{-1}X^{\frac{1}{2}}A_{i}]^{*}W[(\lambda I+X )^{-1}X^{\frac{1}{2}}A_{i}]\lambda^{p_{i}-1}\emph{d}\lambda,
\;\;\;W\in\mathcal{H}^{n\times n}.
 \end{equation} is invertible.
\end{lemma}
\begin{proof} It suffices to show  that the following equation
 \begin{equation*}\label{eq9}
\mathbf{L}W=V
\end{equation*}
has a unique solution for every $V\in \mathcal{H}^{n\times n}$.
Define the operator $\mathbf{M}:\mathcal{H}^{n\times
n}\rightarrow\mathcal{H}^{n\times n}$ by
\begin{equation*}
 \mathbf{M}Z=\sum\limits_{i=1}^{m}\frac{\sin p_{i}\pi}{\pi}\!\int^{\infty}_{0}X^{-\frac{1}{2}}A_{i}^{*}X^{\frac{1}{2}}(\lambda I+X
)^{-1}X^{\frac{1}{2}}ZX^{\frac{1}{2}}(\lambda I+X
)^{-1}X^{\frac{1}{2}}A_{i}X^{-\frac{1}{2}}\lambda^{p_{i}-1}\textmd{d}\lambda,
\;Z\in\mathcal{H}^{n\times n}\!\!.
 \end{equation*}
 Let $Y=X^{-1/2}W X^{-1/2}$. Thus (\ref{eq9}) is equivalent to
$$
Y+\mathbf{M}Y=X^{-1/2}VX^{-1/2}.$$ According to Lemma \ref{lem2}, we
have
\begin{eqnarray*}
||\mathbf{M}Y||&\leq &\sum\limits_{i=1}^{m}\left\|\!\frac{\sin
p_{i}\pi}{\pi}\!\int^{\infty}_{0}\!\!\!\!\!X^{-\frac{1}{2}}A_{i}^{*}X^{\frac{1}{2}}(\lambda
I+X )^{-1}X(\lambda I+X
)^{-1}X^{\frac{1}{2}}A_{i}X^{-\frac{1}{2}}\lambda^{p_{i}-1}\textmd{d}\lambda\right\|||Y||\\
&\leq &\sum\limits_{i=1}^{m}(1-p_{i})||X^{\frac{p_{i}}{2}}A_{i}X^{-\frac{1}{2}}||^{2}||Y||\leq
\sum\limits_{i=1}^{m}\displaystyle\frac{\|A_{i}\|^{2}}{\beta^{1-p_{i}}}||Y||< ||Y||,
\end{eqnarray*}
which implies that $||\mathbf{M}||<1$ and $I+\mathbf{M}$ is
invertible. Therefore, the operator $\mathbf{L}$ is
invertible.\end{proof}

Furthermore, we define operators $ {\textbf P_{i}}:\mathcal{C}^{n\times n}\rightarrow\mathcal{H}^{n\times n}$ by

$${\textbf P_{i}}Z_{\;i}=\textbf{L}^{-1}(B_{i}^{*}Z_{\;i}+Z_{\;i}^{*}B_{i}),\;\;Z_{i}\in\mathcal{C}^{n\times n},\;\;i=1,2, \cdots, m.$$
Thus,we can rewrite (\ref{eq18}) as
\begin{equation}\label{eq19}                                                          
    \Delta X=\textbf{L}^{-1}\Delta Q+\sum\limits_{i=1}^{m}{\textbf P_{i}}\Delta A_{i}+
\textbf{L}^{-1}(\sum\limits_{i=1}^{m}\Delta A_{i}^{*}X^{p_{i}}\Delta A_{i})+\textbf{L}^{-1}(h(\Delta X)).
\end{equation}
Define
$$
||\mathbf{L}^{-1}||=\max_{\begin{array}{c}
W\in\mathcal{H}^{n\times n}\\
||W||=1
\end{array}}||\mathbf{L}^{-1}W||,\;\;\;\;
||\mathbf{P_{i}}||=\max_{\begin{array}{c}
Z\in\mathcal{C}^{n\times n}\\
||Z||=1
\end{array}}||\mathbf{P_{i}}Z||, \;\;i=1, 2, \cdots, m.
$$
 Now we denote
 \begin{eqnarray}\label{eq10}
l&=&\|\textbf{L}^{-1}\|^{-1},\;\;\zeta=\|X^{-1}\|,\;\;\xi_{i}=\|X^{p_{i}}\|,\;\;n_{i}=\|\textbf{P}_{i}\|,\;\;
\theta=\frac{\zeta^{2}}{l}\sum\limits_{i=1}^{m}\xi_{i}\|A_{i}\|^{2},\;\;i=1,2, \cdots, m,\nonumber\\
\epsilon&=& \frac{1}{l}\|\Delta Q\|+\sum\limits_{i=1}^{m}(n_{i}\|\Delta A_{i}\|+\frac{\xi_{i}}{l}\|\Delta A_{i}\|^{2}),\;\;\;\;\sigma\;\;=\;\;\frac{\zeta}{l}\sum\limits_{i=1}^{m}\xi_{i}(2\|A_{i}\|+\|\Delta A_{i}\|)\|\Delta A_{i}\|.
 \end{eqnarray}

  \begin{theorem}\label{thm9}                                         
  If
  \begin{equation}\label{eq20}                                             
   \sigma<1\;\;\mbox{and}\;\;\epsilon<\frac{(1-\sigma)^{2}}
   {\zeta+\sigma\zeta+2\theta+2\sqrt{(\zeta+\theta)(\sigma\zeta+\theta)})},
  \end{equation}
  then
  \begin{equation*}
    \|\widetilde{X}-X\|\leq\frac{2\epsilon}{1+\epsilon\zeta-\sigma+
    \sqrt{(1+\zeta\epsilon-\sigma)^{2}-4\epsilon(\zeta+\theta)}}\equiv\nu.
  \end{equation*}

 \end{theorem}

   \begin{proof}
   {Let \begin{equation}\label{eq12}                                         
   f(\Delta X)=\textbf{L}^{-1}\Delta Q+\sum\limits_{i=1}^{m}{\textbf P_{i}}\Delta A_{i}+
\textbf{L}^{-1}(\sum\limits_{i=1}^{m}\Delta A_{i}^{*}X^{p_{i}}\Delta A_{i})+\textbf{L}^{-1}(h(\Delta X)).
 \end{equation}
 Obviously, $f:\mathcal{H}^{n\times n}\rightarrow\mathcal{H}^{n\times n}$ is continuous. The condition (\ref{eq20}) ensures that the quadratic equation $(\zeta+\theta)x^{2}-(1+\zeta\epsilon-\sigma)x+\epsilon=0$ with respect to the variable $x$ has two positive real roots. The smaller one is
$$\nu=\frac{2\epsilon}{1+\epsilon\zeta-\sigma+
    \sqrt{(1+\zeta\epsilon-\sigma)^{2}-4\epsilon(\zeta+\theta)}}.$$ Define
    $\Omega=\{\Delta X\in \mathcal{H}^{n\times n}: \|\Delta X\|\leq \nu\}.$ Then for any $\Delta X\in \Omega,$ by (\ref{eq20}), we have
    \begin{eqnarray*}
   && ||X^{-1}\Delta X||\leq ||X^{-1}||||\Delta X||\leq\zeta\;\nu\leq
\zeta\cdot\frac{2\epsilon}{1+\zeta\epsilon-\sigma}\\
&&=
1+\frac{\zeta\epsilon+\sigma-1}{1+\zeta\epsilon-\sigma}\leq 1+\frac{-2(1-\sigma)(\zeta\sigma+\theta)}{(\zeta+\sigma\zeta+2\theta)(1+\zeta\epsilon-\sigma)}<1.
\end{eqnarray*}
It follows that $I-X^{-1}\Delta X$ is nonsingular and
$$\|I-X^{-1}\Delta X\|\leq\frac{1}{1-\|X^{-1}\Delta X\|}\leq\frac{1}{1-\zeta\|\Delta X\|}.$$
Using (\ref{eq5}) and Lemma \ref{lem2}, we have
\begin{eqnarray*}
\|Z_{i}(\Delta X)\|&\leq& (1-p_{i})\|\Delta X\|^{2}\|X^{-1}\|^{2}\|(I+X^{-1}\Delta X)^{-1}\|\|X^{p_{i}}\|
\leq \xi_{i}\zeta^{2}\frac{\|\Delta X\|^{2}}{1-\zeta\|\Delta X\|},\\
\|V_{i}(\Delta X)\|&\leq&\|X^{p_{i}}\|\|\Delta X\|\|X^{-1}\|\|(I+X^{-1}\Delta X)^{-1}\|\leq \xi_{i}\zeta\frac{\|\Delta X\|}{1-\zeta\|\Delta X\|},\\
\|h(\Delta X)\|&\leq &\sum\limits_{i=1}^{m}\left(\|A_{i}\|^{2}\|Z_{i}(\Delta X)\|+(2\|A_{i}\|+\|\Delta A_{i}\|)\|\Delta A_{i}\|\|V_{i}(\Delta X)\|\right)\\
&\leq & \sum\limits_{i=1}^{m}\left(\xi_{i}\zeta^{2}\|A_{i}\|^{2}\frac{\|\Delta X\|^{2}}{1-\zeta\|\Delta X\|}+(2\|A_{i}\|+\|\Delta A_{i}\|)\|\Delta A_{i}\|\xi_{i}\zeta\frac{\|\Delta X\|}{1-\zeta\|\Delta X\|}\right).
\end{eqnarray*}

Noting (\ref{eq10}) and (\ref{eq12}), it follows that
\begin{eqnarray*}
\|f(\Delta X)\|&\leq& \frac{1}{l}\|\Delta Q\|+\sum\limits_{i=1}^{m}(n_{i}\|\Delta A_{i}\|+\frac{\zeta_{i}}{l}\|\Delta A_{i}\|^{2})+\frac{1}{l}\|h(\Delta X)\|\\
&\leq &\epsilon+\frac{\sigma\|\Delta X\|}{1-\zeta\|\Delta X\|}+\frac{\theta\|\Delta X\|^{2}}{1-\zeta\|\Delta X\|}\\
&\leq &\epsilon+\frac{\sigma\nu}{1-\zeta\nu}+\frac{\theta\nu^{2}}{1-\zeta\nu}=\nu,
\end{eqnarray*}
for $\Delta X\in \Omega.$ That is $f(\Omega)\subseteq\Omega.$ According to Schauder fixed point theorem, there exists $\Delta X_{*}\in\Omega$ such that $f(\Delta X_{*})=X_{*}.$ It follows that $X+\Delta X_{*}$ is a Hermitian solution of Eq.(\ref{eq7}). By Lemma \ref{lem3}, we know that the solution of Eq.(\ref{eq7}) is unique. Then $\Delta X_{*}=\widetilde{X}-X$ and $\|\widetilde{X}-X\|\leq\xi_{3}.$
}\end{proof}

\begin{rem}   \label{rem3}                                                                   
From Theorem \ref{thm9}, we get the first order perturbation bound for the solution as follows:
\begin{eqnarray*}
&&\|\widetilde{X}-X\|\leq\frac{1}{l}\|\Delta Q\|+\sum\limits_{i=1}^{m}n_{i}\|\Delta A_{i}\|+O\left(\|(\Delta A_{1}, \Delta A_{2}, \cdots, \Delta A_{m}, \Delta Q)\|_{F}^{2}\right),\\\mbox{as}&&(\Delta A_{1}, \Delta A_{2}, \cdots, \Delta A_{m}, \Delta Q)\rightarrow 0.
\end{eqnarray*}
Combining this with (\ref{eq19}) gives
$$\Delta X=\textbf{L}^{-1}\Delta Q+{\textbf L}^{-1}\sum\limits_{i=1}^{m}(B_{i}^{*}\Delta A_{i}+\Delta A_{i}^{*}B_{i})+O\left(\|(\Delta A_{1}, \Delta A_{2}, \cdots, \Delta A_{m}, \Delta Q)\|_{F}^{2}\right).$$
as $\;\;(\Delta A_{1}, \Delta A_{2}, \cdots, \Delta A_{m}, \Delta Q)\rightarrow 0.$
\end{rem}
\section{Backward error of Eq.(\ref{eq1}) with $0<p_{i}<1$}
In this section, a  backward error of an approximate
solution for the unique solution to Eq. (\ref{eq1}) with $0<p_{i}<1$ is developed .

\begin{theorem}\label{thm7}                                                             
Let $\widetilde{X}>0$ be an approximation to the solution $X$ of
Eq.(\ref{eq1}). If $\Sigma=\sum\limits_{i=1}^{m}(1-p_{i})\|\widetilde{X}^{\frac{p_{i}}{2}}A_{i}\widetilde{X}^{-\frac{1}{2}}\|^{2}<1$ and the residual $R(\widetilde{X}) \equiv
Q+\sum\limits_{i=1}^{m}A_{i}^*\widetilde{X}^{p_{i}}A_{i}-\widetilde{X}$ satisfies
\begin{equation}\label{eq11}                                           
\|R(\widetilde{X})\| < \frac{\theta_{1}}{2\|\widetilde{X}^{-1}\|}
    \min\{1,\frac{\theta_{1}}{2}\},
    \;\;\mbox{where}\;\;\theta_{1}\equiv 1+\|\widetilde{X}^{-1}\|\|R(\widetilde{X})\|-\Sigma>0,
\end{equation}
then
$$\|\widetilde{X}-X\|\leq \mu\|R(\widetilde{X})\|,\;\;
\mbox{where}\;\;\mu=\frac{2\|\widetilde{X}\|\|\widetilde{X}^{-1}\|}{\theta_{1}+
\sqrt{\theta_{1}^{2}-4\|\widetilde{X}^{-1}\|\|R(\widetilde{X})\|}}.$$
\end{theorem}

\begin{proof}

Let $$\Psi = \{\Delta X \in
\mathcal{H}^{n\times n}:\parallel \widetilde{X}^{-1/2}\Delta X
\widetilde{X}^{-1/2}\parallel\leq \theta_{2}
\|R(\widetilde{X})\|\},$$ where $\theta_{2}=\frac{\mu}{\|\widetilde{X}\|}.$ Obviously, $\Psi$ is a nonempty
bounded convex closed set. Let $$g(\Delta
X)=\sum\limits_{i=1}^{m}A_{i}^*\left[(\widetilde{X}+\Delta
X)^{p_{i}}-\widetilde{X}^{p_{i}}\right]A_{i}+R(\widetilde{X}).$$ Evidently $g:
\Psi\mapsto\mathcal{H}^{n\times n}$ is continuous. We will prove
that $g(\Psi)\subseteq\Psi$. For every $\Delta X\in\Psi,$ we have
$$\|\widetilde{X}^{-1/2}\Delta X\widetilde{X}^{-1/2}\|\leq \theta_{2}\|R(\widetilde{X})\|.$$ Hence
 $$\widetilde{X}^{-1/2}\Delta X\widetilde{X}^{-1/2}\geq -\theta_{2}\|R(\widetilde{X})\|I,$$ that is
 $$\widetilde{X}+\Delta X\geq (1-\theta_{2}\|R(\widetilde{X})\|)\widetilde{X}.$$ Using (\ref{eq11}), one sees
that
$$\theta_{2}\|R(\widetilde{X})\|=\frac{2\|\widetilde{X}^{-1}\|\|R(\widetilde{X})\|}{\theta_{1}+
\sqrt{\theta_{1}^{2}-4\|\widetilde{X}^{-1}\|\|R(\widetilde{X})\|}}<
\frac{2\|\widetilde{X}^{-1}\|\|R(\widetilde{X})\|}{\theta_{1}}<1.$$
Therefore, $(1-\theta_{2}\|R(\widetilde{X})\|)\widetilde{X}>0.$

According to (\ref{eq16}), we obtain
\begin{eqnarray*}
&&\parallel \widetilde{X}^{-\frac{1}{2}}g(\Delta
X)\widetilde{X}^{-\frac{1}{2}}\parallel\\
&\leq&\left(\|\widetilde{X}^{-\frac{1}{2}}\Delta X
\widetilde{X}^{-\frac{1}{2}}\|+\frac{\|\widetilde{X}^{-\frac{1}{2}}\Delta X
\widetilde{X}^{-\frac{1}{2}}\|^{2}}{1-\theta_{2}\|R(\widetilde{X})\|}\right)\Sigma+
\|\widetilde{X}^{-\frac{1}{2}}R(\widetilde{X})\widetilde{X}^{-\frac{1}{2}}\|\\
&\leq &\left(\theta_{2}\|R(\widetilde{X})\|+\frac{(\theta_{2}\|R(\widetilde{X})\|)^{2}}{1-\theta_{2}\|R(\widetilde{X})\|}\right)\Sigma
+\|\widetilde{X}^{-1}\|\|R(\widetilde{X})\|\\
&=&\theta_{2}\|R(\widetilde{X})\|.
\end{eqnarray*}
By Brouwer fixed point theorem, there exists a $\Delta X\in\Psi$
such that $g(\Delta X)=\Delta X.$ Hence $\widetilde{X}+\Delta X$
is a solution of Eq.(\ref{eq1}). Moreover, by Lemma \ref{lem3}, we
know that the solution $X$ of Eq.(\ref{eq1}) is unique. Then
\begin{equation*}
    \|\widetilde{X}-X\|=\|\Delta X\|\leq\|\widetilde{X}\|\|\widetilde{X}^{-\frac{1}{2}}\Delta X\widetilde{X}^{-\frac{1}{2}}\|
    =\theta_{2}\|\widetilde{X}\|\|R(\widetilde{X})\|.
\end{equation*}
\end{proof}

\section{Condition number}
In this section, we apply the theory of condition number developed by Rice \cite{s37}
to study condition numbers of the unique solution to Eq. (\ref{eq1}) with $0<p_{i}<1$.
\subsection{The complex case }

Suppose that $X$ and $\widetilde{X}$ are the solutions of the
matrix equations (\ref{eq1}) and (\ref{eq7}), respectively. Let
$\Delta A=\widetilde{A}-A$, $\Delta Q=\widetilde{Q}-Q$ and $\Delta
X=\widetilde{X}-X$. Using Theorem \ref{thm9} and Remark \ref{rem3}, we have

\begin{equation}\label{eq21}                                                                    
\Delta X=\widetilde{X}-X=\textbf{L}^{-1}\Delta Q+{\textbf L}^{-1}\sum\limits_{i=1}^{m}(B_{i}^{*}\Delta A_{i}+\Delta A_{i}^{*}B_{i})+O\left(\|(\Delta A_{1}, \Delta A_{2}, \cdots, \Delta A_{m}, \Delta Q)\|_{F}^{2}\right),
\end{equation}
as $(\Delta A_{1}, \Delta A_{2}, \cdots, \Delta A_{m}, \Delta Q)\rightarrow 0.$

By the theory of condition number developed by Rice \cite{s37}, we define
the condition number of the Hermitian positive definite solution
$X$ to (\ref{eq1}) by
\begin{equation}  \label{eq22}                                               
c(X)=\lim_{\delta\rightarrow 0}\sup_{||(\frac{\Delta
A_{1}}{\eta_{1}}, \frac{\Delta
A_{2}}{\eta_{2}}, \cdots, \frac{\Delta
A_{m}}{\eta_{m}},
\frac{\Delta Q}{\rho})||_{F}\leq\delta}\frac{||\Delta
X||_{F}}{\xi\delta},
\end{equation}
where $\xi$,  $\rho$ and $\eta_{i},$ $i=1,2, \cdots, m$ are positive parameters. Taking
$\xi=\eta_{i}=\rho=1$  in (\ref{eq22}) gives the absolute condition number $c_{abs}(X)$,
and taking $\xi=||X||_{F}$, $\eta_{i}=||A_{i}||_{F}$  and $\rho=||Q||_{F}$ in (\ref{eq22})
gives the relative condition number $c_{rel}(X)$.

Substituting (\ref{eq21}) into (\ref{eq22}), we get
\begin{eqnarray*}  \label{eq23}                                             
c(X)&=&\frac{1}{\xi}\!\!\!\max_{\begin{array}{c}(\frac{\Delta
A_{1}}{\eta_{1}}, \frac{\Delta
A_{2}}{\eta_{2}}, \cdots, \frac{\Delta
A_{m}}{\eta_{m}},
\frac{\Delta Q }{\rho})\neq 0\\
 \Delta A_{i}\in\mathcal{C}^{n\times
n}, \Delta Q\in\mathcal{H}^{n\times n}
\end{array}}
\!\!\!\!\!\!\!\!\frac{||\mathbf{L}^{-1}(\Delta Q+\sum\limits_{i=1}^{m}(B_{i}^{*}\Delta
A_{i}+\Delta A_{i}^{*}B_{i})) ||_{F}}{||(\frac{\Delta
A_{1}}{\eta_{1}}, \frac{\Delta
A_{2}}{\eta_{2}}, \cdots, \frac{\Delta
A_{m}}{\eta_{m}},, \frac{\Delta
Q}{\rho})||_{F}}\nonumber\\
&=&\frac{1}{\xi}\!\!\!\max_{\begin{array}{c}(E_{1}, E_{2}, \cdots, E_{m},
H)\neq 0\\
 E_{i}\in\mathcal{C}^{n\times
n}, H\in\mathcal{H}^{n\times n}
\end{array}}
\!\!\!\!\!\!\!\!\frac{||\mathbf{L}^{-1}(\rho H+\sum\limits_{i=1}^{m}\eta_{i}(B_{i}^{*}E_{i}+E_{i}^{*}B_{i}))
||_{F}}{||(E_{1}, E_{2}, \cdots, E_{m}, H)||_{F}}.
\end{eqnarray*}
Let $L$ be the matrix representation of the linear operator
$\mathbf{L}$. Then it is easy to see that
\begin{equation*}
L=I+\sum\limits_{i=1}^{m}\frac{\sin p_{i}\pi}{\pi}\int_{0}^{\infty}[(\lambda I+X)^{-1}X^{\frac{1}{2}}A^{}_{i}]^{T}\otimes [(\lambda I+X)^{-1}X^{\frac{1}{2}}A^{}_{i}]^{*}\lambda^{p_{i}-1}d\lambda.
\end{equation*}
Let\begin{eqnarray}\label{eq26}
&&L^{-1}=S+\textbf{i}\Sigma,\nonumber\\
&&L^{-1}(I\otimes B_{i}^{*})=L^{-1}(I\otimes
(X^{p_{i}}A_{i})^{*})=U_{i1}+\textbf{i}\Omega_{i1},\nonumber\\
&&L^{-1}(B_{i}^{T}\otimes I)\Pi=L^{-1}((X^{p_{i}}A_{i})^{T}\otimes
I)\Pi=U_{i2}+\textbf{i}\Omega_{i2},\nonumber\\
&& S_{c}=\left[\begin{array}{cc}
S & -\Sigma \\
\Sigma & S
\end{array}\right],\;\;\;\;
U_{i}=\left[\begin{array}{cc}
U_{i1}+U_{i2} & \Omega_{i2}-\Omega_{i1}\\
\Omega_{i1}+\Omega_{i2} & U_{i1}-U_{i2}\end{array}\right],\;\;i=1,2, \cdots, m,
\end{eqnarray}
$$\vector H=x+\textbf{i}y,\;\;\vector E_{i}=a_{i}+\textbf{i}b_{i},
\;\;g=(x^{T}, y^{T}, a^{T}_{1}, b^{T}_{1}, \cdots, a^{T}_{m}, b^{T}_{m})^{T},\;M=(E_{1}, E_{2}, \cdots, E_{m},H),$$
where $x, y, a_{i}, b_{i}
\in\mathcal{R}^{n^{2}},\; S, \Sigma, U_{i1}, U_{i2}, \Omega_{i1}, \Omega_{i2}\in\mathcal{R}^{n^{2}\times n^{2}},\;i=1, 2, \cdots, m, $ $\Pi$ is the vec-permutation matrix, such that
$$\vector \;A^{T}=\Pi \;\vector \;A.$$
Then we obtain that
\begin{eqnarray*}
&&c(X)=\frac{1}{\xi}\max_{\begin{array}{c}M\neq 0\\
\end{array}}
\frac{||\mathbf{L}^{-1}(\rho H+\sum\limits_{i=1}^{m}\eta_{i}(B_{i}^{*}E_{i}+E_{i}^{*}B_{i}))
||_{F}}{||(E_{1}, E_{2}, \cdots, E_{m}, H)||_{F}}\\
&=&\frac{1}{\xi}\!\!\!\max_{\begin{array}{c}M\neq 0\\
\end{array}}
\!\!\frac{||\rho{L}^{-1} \vector H+\sum\limits_{i=1}^{m}\eta_{i}{L}^{-1}((I\otimes B_{i}^{*})\vector E_{i}+(B_{i}^{T}\otimes I)\vector E_{i}^{*})
||}{\left\|\left(
\vector E_{1},
\vector E_{2},
\cdots,
\vector E_{m}, \vector H
\right)\right\|}\\
&=&\frac{1}{\xi}\!\!\!\max_{\begin{array}{c}M\neq 0\\
\end{array}}
\!\!\!\!\frac{||\rho(S+\textbf{i}\Sigma)(x+\textbf{i}y)+
\sum\limits_{i=1}^{m}\eta_{i}[(U_{i1}+\textbf{i}\Omega_{i1})(a_{i}+\textbf{i}b_{i})
||+(U_{i2}+\textbf{i}\Omega_{i2})(a_{i}-\textbf{i}b_{i})
]||}{\left\|\left(
\vector E_{1},
\vector E_{2},
\cdots,
\vector E_{m}, \vector H
\right)\right\|}\\
&=&\frac{1}{\xi}\!\!\!\max_{\begin{array}{c}g\neq 0\\
\end{array}}
\frac{||(\rho\; S_{c}, \eta_{1}U_{1}, \eta_{2}U_{2}, \cdots, \eta_{m}U_{m})g
||}{\|g\|}\\
&=&\frac{1}{\xi}\;||\;(\rho S_{c},\;\eta_{1}U_{1},\;\eta_{2}U_{2}, \cdots, \eta_{m} U_{m})||,\;\;E_{i}\in\mathcal{C}^{n\times
n}, H\in\mathcal{H}^{n\times n}.
\end{eqnarray*}

Then we have the following theorem.
\begin{theorem}\label{thm10}                                           
The condition number $c(X)$ defined by (\ref{eq22}) has the
explicit expression
\begin{equation}\label{eq31}                                        
c(X)=\frac{1}{\xi}\;||\;(\rho S_{c},\;\eta_{1}U_{1},\;\eta_{2}U_{2}, \cdots, \eta_{m} U_{m})||,
\end{equation}
where the matrices $S_{c}$ and $U_{i}$ are defined as in (\ref{eq26}).
\end{theorem}
\begin{rem}\label{rem4}                                          
From (\ref{eq31}) we have the relative condition number
\begin{equation*} \label{eq32}                                               
c_{rel}(X)=\frac{||\;(||Q||_{F}S_{c},\;||A_{1}||_{F}
U_{1},\;||A_{2}||_{F}
U_{2},\;\cdots, ||A_{m}||_{F}
U_{m})||}{||X||_{F}}.
\end{equation*}
\end{rem}
\subsection{The real case  }
In this subsection we consider the real case, i.e.,
all the coefficient matrices $A_{i}$, $Q$ of Eq.(\ref{eq1}) are real. In
such a case the corresponding solution $X$ is also real.
Completely similar arguments as Theorem \ref{thm10} gives the following
theorem.
\begin{theorem}
Let $A_{i}$, $Q$ be real, $c(X)$ be the condition number defined by
(\ref{eq22}). Then $c(X)$ has the explicit expression
\begin{equation*}
c(X)=\frac{1}{\xi}\;||\;(\rho S_{r},\;\eta_{1}U_{1},\;\eta_{2}U_{2}, \cdots, \eta_{m}U_{m})\;||,
\end{equation*}
where
\begin{eqnarray*}
&&S_{r}=\left(I+\sum\limits_{i=1}^{m}\frac{\sin p_{i}\pi}{\pi}\int_{0}^{\infty}[(\lambda I+X)^{-1}X^{\frac{1}{2}}A^{}_{i}]^{T}\otimes [(\lambda I+X)^{-1}X^{\frac{1}{2}}A^{}_{i}]^{*}\lambda^{p_{i}-1}d\lambda\right)^{-1},\\
&&U_{i}=S_{r}[I\otimes(A_{i}^{T}X^{p_{i}})+((A_{i}^{T}X^{p_{i}})\otimes I)\Pi],\;\;i=1, 2, \cdots, m.
\end{eqnarray*}
\end{theorem}

\begin{rem}\label{rem2}                                              
In the real case the relative condition number is given by
\begin{equation*}
c_{rel}(X)=\frac{||\;(||Q||_{F}S_{r},\;||A_{1}||_{F}U_{1},\;||A_{2}||_{F}U_{2}, \cdots,||A_{m}||_{F}
U_{m}||}{||X||_{F}}.
\end{equation*}
\end{rem}
\section{Numerical Examples }
To illustrate the results of the previous sections, in this
section three simple examples are given, which were carried out
using MATLAB 7.1. For the stopping criterion we take
$\varepsilon_{k+1}(X)=\|X-\sum\limits_{i=1}^{m}A_{i}^{*}X^{p_{i}}A_{i}-I\|<1.0e-10$


\begin{example}        \label{ex1}                                                   
We consider the matrix equation $$X-A_{1}^{*}X^{\frac{1}{2}}A_{1}-A_{2}^{*}X^{\frac{1}{3}}A_{2}=I,$$
with \[A_{1}=\frac{\frac{1}{3}+2\times 10^{-2}}{||A||}A, \;\;A_{2}=\frac{\frac{1}{6}+3\times 10^{-2}}{||A||}A,\;\;\;A=\left(\begin{array}{cc}
2 & 0.95\\
0 & 1
\end{array}\right).
\]
Suppose that the
coefficient matrices $A_{1}$ and $A_{2}$ are perturbed to
$\widetilde{A_{i}}=A_{i}+\Delta A_{i},i=1,2$, where $$\Delta A_{1}=\frac{10^{-j}}{\|C^{T}+C\|}(C^{T}+C),\;\;\Delta A_{2}=\frac{3\times10^{-j-1}}{\|C^{T}+C\|}(C^{T}+C)$$ and $C$ is a random matrix generated by
MATLAB function \textbf{randn}.

We now consider the corresponding  perturbation bounds for the solution $X$ in Theorem \ref{thm6} and Theorem \ref{thm9}.

The conditions in Theorem \ref{thm6} are
\begin{eqnarray*}
con1&=& 2(\beta-s)-b>0,\;\;con2=\beta+\|\Delta Q\|-\sum\limits_{i=1}^{m}(1-{p_{i}})\beta^{p_{i}}\|A_{i}\|^{2}>0,\\
  con3 &=& b^{2}-4(\beta-s)(s+\|\Delta Q\|)\geq 0.
\end{eqnarray*}

The conditions in Theorem \ref{thm9} are
\begin{eqnarray*}
con4&=& 1-\sigma>0,\;\;
  con5 = \frac{(1-\sigma)^{2}}
   {\zeta+\sigma\zeta+2\theta+2\sqrt{(\zeta+\theta)(\sigma\zeta+\theta)})}-\epsilon>0.
\end{eqnarray*}
By computation, we list them in Table \ref{tab:1}.
\begin{table}[h t b]
\caption{Conditions for Example \ref{ex1} with different values of j}
\label{tab:1}
\begin{tabular}{p{1.5cm}p{2.4cm}p{2.4cm}p{2.4cm}p{2.4cm}}
\hline\noalign{\smallskip}
$j$                & 4             & 5              & 6              & 7\\
\noalign{\smallskip}\hline\noalign{\smallskip}
$con1$ & $1.1139$ & $1.1141$ & $1.1141$ & $1.1141$  \\
$con2$ & $0.9358$ & $0.9358$ & $ 0.9357$ & $0.9357$\\
$con3$ & $0.8751$ & $0.8755$ & $ 0.8756$ & $0.8756$\\
$con4$ & $1.0000$ & $1.0000$ & $1.0000$ & $1.0000$\\
$con5$ & $0.7955$ & $0.7957$ & $ 0.7957$ & $0.7957$\\ \hline
\noalign{\smallskip}
\end{tabular}
\end{table}

The results listed in Table \ref{tab:1} show that the conditions of Theorem \ref{thm6}
and Theorem \ref{thm9} are satisfied.
\vskip 0.1in
By Theorem \ref{thm6} and Theorem \ref{thm9}, we can
compute the relative perturbation bounds $\xi_{1}, \xi_{2}=\frac{\nu}{\|X\|},$  respectively. These results averaged as the geometric mean of  10 randomly perturbed runs.
Some results are listed in Table \ref{tab:2}.
\begin{table}[h t b]
\caption{Results for Example \ref{ex1} with different values of j}
\label{tab:2}
\begin{tabular}{p{1.5cm}p{2.4cm}p{2.4cm}p{2.4cm}p{2.4cm}}
\hline\noalign{\smallskip}
$j$                & 4             & 5              & 6              & 7\\
\noalign{\smallskip}\hline\noalign{\smallskip}
$\frac{\|\widetilde{X}-X\|}{\|X\|}$ & $3.9885\times 10^{-5}$ & $5.1141\times 10^{-6}$ & $3.6513\times10^{-7}$ & $ 4.6136\times10^{-8}$  \\
 $\xi_{1}$ &  $
1.9765\times 10^{-4}$ & $ 2.3869\times 10^{-5}$ & $ 1.8133\times
10^{-6}$ & $2.1028\times10^{-7}$\\
 $\xi_{2}$ &  $
6.5069\times 10^{-5}$ & $ 7.6524\times 10^{-6}$ & $6.0514\times
10^{-7}$ & $6.9911\times10^{-8}$\\ \hline
\noalign{\smallskip}
\end{tabular}
\end{table}

The results listed in Table \ref{tab:2} show that the perturbation
bound $\xi_{2}$ given by Theorem \ref{thm9} is fairly sharp, while the bound $\xi_{1}$ given by Theorem \ref{thm6} which does not depended on the exact solution is conservative.
\end{example}
\begin{example}\label{ex2}
We consider the matrix equation
 $$X-A_{1}^{*}X^{0.5}A_{1}-A_{2}^{*}X^{0.25}A_{2}=I,$$
with \[A_{1}=\frac{\frac{1}{3}+2\times 10^{-2}}{||A||}A, \;\;A_{2}=\frac{\frac{1}{6}+3\times 10^{-2}}{||A||}A,\;\;\;A=\left(\begin{array}{ccccc}
2 & 1& 0& 0& 0\\
1 & 2& 1& 0& 0\\
0 & 1& 2& 1& 0\\
0 & 0& 1& 2& 1\\
0 & 0& 0& 1& 2
\end{array}\right).
\]                                                          
Choose $\widetilde{X}_1=A$, $\widetilde{X}_2=2A$. Let the
approximate solution $\widetilde{X}_k$ of $X$ be given with the
iterative method (\ref{eq13}), where $k$ is the iterative number.

The residual $R(\widetilde{X}_k) \equiv
I+A_{1}^*\widetilde{X}_k^{0.5}A_{1}+A_{2}^*\widetilde{X}_k^{0.25}A_{2}-\widetilde{X}_k$ satisfies the conditions in
Theorem \ref{thm7}.

By Theorem \ref{thm7}, we can compute the backward error
bound for $\widetilde{X}_k$
$$\parallel \widetilde{X}_k-X\parallel \leq
\mu\|R(\widetilde{X}_k)\|,$$ 
where $$\mu=\frac{2\|\widetilde{X}_{k}\|\|\widetilde{X}_{k}^{-1}\|}{\theta_{1}+
\sqrt{\theta_{1}^{2}-4\|\widetilde{X}_{k}^{-1}\|\|R(\widetilde{X}_{k})\|}},\;\;
\theta_{1}\equiv1+\|\widetilde{X}_{k}^{-1}\|\|R(\widetilde{X}_{k})\|-
(0.5\|\widetilde{X_{k}}^{\frac{1}{4}}A_{1}\widetilde{X_{k}}^{-\frac{1}{2}}\|^{2}+0.75
\|\widetilde{X_{k}}^{\frac{1}{8}}A_{2}\widetilde{X_{k}}^{-\frac{1}{2}}\|^{2}).$$  Some results are
listed in Table\ref{tab:3}.
\begin{table}[h t b]
\caption{Results for Example \ref{ex2} with different values of k}
\label{tab:3}
\begin{tabular}{p{1.5cm}p{2.4cm}p{2.4cm}p{2.4cm}p{2.4cm}}
\hline\noalign{\smallskip}
$k$           & 8     &10             & 12              & 14            \\
\noalign{\smallskip}\hline\noalign{\smallskip}
  $||\widetilde{X}_k-X||$& $6.1091\times 10^{-4}$ &$4.0865\times 10^{-5}$ & $  2.6837\times 10^{-6}$& $1.7372\times 10^{-7}$\\ \hline
 $\mu|| R(\widetilde{X}_k)||$ & $  7.2094\times 10^{-4}    $&$4.8224\times 10^{-5}$ & $3.1670\times 10^{-6}$ &$2.0506\times 10^{-7}$ \\ \hline
\noalign{\smallskip}
\end{tabular}
\end{table}

The results listed in Table \ref{tab:3} show that the error bound
given by Theorem \ref{thm7} is fairly sharp.
\end{example}

\begin{example}               \label{ex3}                                               
We study the matrix equation
 $$X-A_{1}^{*}X^{\frac{1}{2}}A_{1}-A_{2}^{*}X^{\frac{1}{3}}A_{2}=Q,$$
with \[A_{1}=\left(\begin{array}{cc}
0 & 0.55+10^{-k}\\
0 & 0
\end{array}\right),\;\;A_{2}=\frac{1}{2}A_{1},\;\;Q=\left(\begin{array}{cc}
1 & 1\\
0 & 1
\end{array}\right).
\]
By Remark \ref{rem3}, we can compute the relative
condition number $c_{rel}(X).$ Some results are listed in Table
\ref{tab:4}.
\begin{table}[h t b]
\caption{Results for Example \ref{ex3} with different values of $k$}
\label{tab:4}
\begin{tabular}{p{2cm}p{1.7cm}p{1.7cm}p{1.7cm}p{1.7cm}p{1.7cm}}
\hline\noalign{\smallskip}
  $k$ & 1 & 3 & 5 & 7 & 9 \\ \noalign{\smallskip}\hline\noalign{\smallskip}
 $c_{rel}(X)$& 1.1888 &  1.1025 & 1.1019 &  1.1019 &  1.1019 \\ \hline\noalign{\smallskip}
\end{tabular}
\end{table}

The numerical results listed in the second line show that the unique
positive definite solution $X$ is well-conditioned.
\end{example}
\section*{Acknowledgements}
The author wishes to express her gratitude to the referees for their fruitful comments
and suggestions regarding the earlier version of this paper.


\begin{thebibliography}{00}

\bibitem{s2}
W. N. Anderson, G. B. Kleindorfer, M. B. Kleindorfer, and
M. B. Woodroofe, Consistent estimates of the parameters of a linear
systems, Ann. Math. Statist. 40 (1969) 2064--2075.

\bibitem{s6}
W. N. Anderson, T. D. Morley, and G. E. Trapp, The cascade
limit, the shorted operator and quadratic optimal control, in
Linear Circuits, Systems and Signal Processsing: Theory and
Application (Christopher I. Byrnes, Clyde. F. Martin, and Richard
E. Saeks, Eds.), North-Holland, New York, 1988, 3--7.

\bibitem{s3}
R. S. Bucy, A priori bounds for the Riccati equation,
in proceedings of the Berkley Symposium on Mathematical Statistics
and Probability, Vol. III: Probability Theory, Univ. of California Press, Berkeley, 1972, 645--656.

\bibitem{s7}
B. L. Buzbee, G. H. Golub, C. W. Nielson, On direct methods for solving Poisson's
equations, SIAM J. Numer. Anal. 7 (1970)627-656.

\bibitem{s21}
J. Cai, G.L. Chen, On the Hermitian positive definite solution of nonlinear matrix equation
$X^{s}+A^{*}X^{-t}A=Q,$ Appl. Math. Comput. 217 (2010) 2448--2456.

\bibitem{s19}
X.F. Duan, A.P. Liao, on the existence of Hermitian positive definite solutions of the
matrix equation $X^{s}+A^{*}X^{-t}A=Q,$ Linear Algebra Appl. 429 (2008) 673--687.

\bibitem{s29}
X.F. Duan, A.P. Liao, B.Tang, On the nonlinear matrix equation $X-\sum\limits_{i=1}^{m}A_{i}^{*}X^{\delta_{i}}A_{i}=Q$, Linear Algebra Appl. 429 (2008) 110-121.

\bibitem{s31}
X.F. Duan, C.M. Li, A.P. Liao, Solutions and perturbation analysis for the nonlinear matrix equation $X+\sum\limits_{i=1}^{m}A_{i}^{*}X^{-1}A_{i}=I$, Appl. Math. Comput. 218 (2011) 4458-4466.

\bibitem{s10}
J.C. Engwerda, On the existence of a positive definite
solution of the matrix equation $X+A^TX^{-1}A=I$, Linear Algebra
Appl. 194 (1993) 91--108.

\bibitem{s11}
J.C. Engwerda, A.C.M. Ran, A.L. Rijkeboer,
Necessary and sufficient conditions for the existence of a positive
definite solution of the matrix equation $X+A^TX^{-1}A=Q$, Linear
Algebra Appl. 186 (1993) 255--275.

\bibitem{s35}
B.R. Fang, J.D. Zhou, Y.M. Li, Matrix Theory, Tsinghua University Press, Beijing, 2006.

\bibitem{s13}
A. Ferrante, Hermitian solutions of the equation
$X=Q+NX^{-1}N^*$, Linear Algebra Appl. 247 (1996) 359--373.


\bibitem{s12}
C.H. Guo, P. Lancaster, Iterative solution of two matrix
equations, Math. Comp. 68 (1999) 1589--1603.

\bibitem{s28}
V.I. Hasanov, Positive definite solutions of the
matrix equations $X\pm A^*X^{-q}A=Q$, Linear Algebra Appl.
404 (2005) 166--182.

\bibitem{s25}
V.I. Hasanov, I.G. Ivanov, On two perturbation
estimates of the extreme solutions to the equations $X\pm
A^*X^{-1}A=Q$, Linear Algebra Appl. 413 (2006), 81--92.

\bibitem{s24}
V.I. Hasanov, I.G. Ivanov, F. Uhlig, Improved
perturbation estimates for the matrix equation $X\pm A^*X^{-1}A=Q$,
Linear Algebra Appl. 379 (2004) 113--135.


\bibitem{s30}
Y.M. He, J.H. Long, On the Hermitian positive definite solution of the nonlinear matrix equation $X+\sum\limits_{i=1}^{m}A_{i}^{*}X^{-1}A_{i}=I$, Appl. Math. Comput. 216 (2010) 3480-3485.

\bibitem{s14}
I.G. Ivanov, S.M. El-Sayed, Properties of positive definite
solution of the equation $X+A^*X^{-2}A=I$, Linear Algebra Appl.
279(1998) 303--316.

\bibitem{s15}
I.G. Ivanov, V.I. Hasanov, B.V. Minchev, On matrix
equations $X\pm A^*X^{-2}A=I$, Linear Algebra Appl. 326 (2001) 27--44.

\bibitem{s38}
Guanjun Jia, Dongjie Gao, Perturbation estimates for the nonlinear matrix equation
$X- A^*X^{q}A=Q$ $(0<q<1),$ J. Appl. Math. Comput. 35 (2011) 295-304.

\bibitem{s26}
J. Li, The Hermitian positive
definite solutions and perturbation analysis of the matrix equation
$X- A^*X^{-1}A=Q$, Math. Numer. Sinica, 30 (2008) 129--142. (in Chinese).

\bibitem{s27}
J. Li, Y.H. Zhang, Perturbation analysis of the matrix equation $X-A^*X^{-q}A=Q$, Linear Algebra Appl.
431 (2009) 1489--1501.


\bibitem{s18}
X.G. Liu, H. Gao, On the positive definite solutions of the matrix equations
$X^{s}\pm A^{T}X^{-t}A=I_{n}$, Linear Algebra Appl.
368 (2003) 83--97.

\bibitem{s4}
D. V. Ouellette, Schur complements and statistics,
Linear Algebra Appl. 36 (1981) 187--295.

\bibitem{s5}
W. Pusz and S. L. Woronowitz, Functional caculus for
sequlinear forms and purification map, Rep. Math. Phys.
8 (1975)159--170.

\bibitem{s36}
M.C.B. Reurings, Symmetric Matrix Equations, The Netheerlands: Universal Press, 2003.

\bibitem{s37}
J.R. Rice, A theory of condition, SIAM J. Numer. Anal. 3 (1966) 287-310.

\bibitem{s23}
J.G. Sun, S.F. Xu, perturbation analysis of the maximal
solution of the matrix equation $X+A^*X^{-1}A=P$.  II, Linear
Algebra Appl. 362 (2003) 211--228.

\bibitem{s32}
S.G. Wang, M.X. Wu, Z.Z. Jia, Matrix inequalities, Science Press, Beijing, 2006.

\bibitem{s34}
J.F. Wang, Y.H.Zhang, B.R. Zhu, The Hermitian positive
definite solutions of the matrix equation
$X+A^*X^{-q}A=I (q>0)$, Math. Numer. Sinica, 26 (2004) 61--72. (in Chinese).

\bibitem{s22}
S.F. Xu, Perturbation analysis of the maximal
solution of the matrix equation $X+A^*X^{-1}A=P$, Linear Algebra
Appl. 336 (2001) 61--70.

\bibitem{s20}
X.Y. Yin, S.Y. Liu, L. Fang, Solutions and perturbation estimates for the matrix equation
$X^{s}+A^{*}X^{-t}A=Q,$ Linear Algebra Appl. 431 (2009) 1409-1421.

\bibitem{s39}
X.Y. Yin, S.Y. Liu and T.X. Li, ON positive definite solutions of the matrix equation
$X+A^*X^{-q}A=Q\;(0<q\leq 1),$ TAIWAN J. Math. 16 (2012) 1391--1407.

\bibitem{s1}
J. Zabezyk, Remarks on the control of discrete time distributed parameter systems,
SIAM J. Control. 12 (1974) 721--735.


\bibitem{s9}
X. Zhan, Computing the extremal positive definite solutions of a
matrix equations, SIAM J. Sci. Comput. 17 (1996) 1167--1174.

\bibitem{s33}
X. Zhan, Matrix Inequalities, Springer-Verlag, Berlin, 2002.

\bibitem{s8}
X. Zhan, J. Xie, On the matrix equation
$X+A^*X^{-1}A=I$, Linear Algebra Appl. 247 (1996) 337--345.

\bibitem{s16}
Y.H. Zhang, On Hermitian positive definite solutions
of matrix equation $X+A^*X^{-2}A=I$, Linear Algebra Appl.
372 (2003) 295--304.

\bibitem{s17}
Y.H. Zhang, On Hermitian positive definite solutions of matrix
equation $X- A^*X^{-2}A=I$, J. Comp. Math. 23 (2005) 408--418.


















\end{thebibliography}
\end{document}